\documentclass[12pt]{amsart}

\usepackage{hyperref}
\usepackage{graphicx}
\usepackage[ansinew]{inputenc}
\usepackage{enumerate}
\usepackage{amsmath}
\usepackage{amsfonts}
\usepackage{amssymb}
\usepackage{amsthm}
\usepackage{multirow}
\usepackage{graphicx}
\usepackage{color}
\usepackage{microtype}
\input xy
\xyoption{all}
\usepackage{url}
\usepackage{varioref}
\usepackage{eso-pic}
\usepackage{tikz}
\usepackage{float}
\usepackage{dsfont}
\usepackage{ifthen}
\usepackage{xparse}
\usepackage[nomessages]{fp}
\usepackage{listings}
\usepackage{xstring}

\lstdefinelanguage{Magma}
{
keywords={for,end,if,then,else,elif,while,function,return,cat,&,and,or },
morekeywords={Seqset,Setseq,Polytope,AutomorphismGroup,RowSequence,IdentifyGroup,
	      Subgroups,PermutationMatrix,Generators,MatrixGroup,Transpose},
sensitive=false,
morecomment=[l]{//},
morecomment=[s]{/*}{*/},
morestring=[b]",
}

\lstnewenvironment{codice_magma}[1][]
{\lstset{basicstyle=\small\ttfamily, columns=fullflexible,
language=Magma,
keywordstyle=\color{red}\bfseries,
commentstyle=\color{blue},
numbers=left, numberstyle=\tiny,
stepnumber=2, numbersep=5pt, frame=single, #1}}{}

\usepackage{multicol}
\newtheorem{thm}{Theorem}
\newtheorem*{thm*}{Theorem}
\newtheorem{lem}[thm]{Lemma}
\newtheorem{prop}[thm]{Proposition}

\newtheorem{conj}[thm]{Conjecture}

\theoremstyle{definition}

\theoremstyle{remark}
\newtheorem{remark}[thm]{Remark}

\DeclareMathOperator{\Pic}{Pic}

\DeclareMathOperator{\PGL}{PGL}

\renewcommand{\H}{H}
\newcommand{\h}{h}
\newcommand{\N}{\mathbb{N}}

\newcommand{\PP}{\mathbb{P}}
\newcommand{\PS}{\mathbb{S}}

\newcommand{\OO}{\mathcal{O}}
\newcommand{\I}{\mathcal{I}}
\newcommand{\M}{\mathcal{M}}
\newcommand{\LL}{\mathcal{L}}

\newcommand{\CICYu}[2]{
\FPeval{\res}{clip(#1-1)}
\ifthenelse{\equal{\res}{1}}{\mkern-9mu\bigstar\mkern-9mu}{}
\left(\begin{array}{c|c}
\PP^{\res} & {#1} \\ 	\FPeval{\res}{clip(#2-1)}
\PP^{\res} & {#2}
\end{array}\right)
}

\newcommand{\CICYd}[4]{
\FPeval{\res}{clip(#1+#2-1)}
\ifthenelse{\equal{\res}{1}}{\mkern-9mu\bigstar\mkern-9mu}{}
\left(\begin{array}{c|cc}
\PP^{\res} & {#1} & {#2} \\ 	\FPeval{\res}{clip(#3+#4-1)}
\PP^{\res} & {#3} & {#4}
\end{array}\right)}

\newcommand{\CICYt}[6]{
\FPeval{\res}{clip(#1+#2+#3-1)}
\ifthenelse{\equal{\res}{1}}{\mkern-9mu\bigstar\mkern-9mu}{}
\left(\begin{array}{c|ccc}
\PP^{\res} & {#1} & {#2} & {#3} \\ 	\FPeval{\res}{clip(#4+#5+#6-1)}
\PP^{\res} & {#4} & {#5} & {#6}
\end{array}\right)}

\newcommand{\CICYq}[8]{
\FPeval{\res}{clip(#1+#2+#3+#4-1)}
\ifthenelse{\equal{\res}{1}}{\mkern-9mu\bigstar\mkern-9mu}{}
\left(\begin{array}{c|cccc}
\PP^{\res} & {#1} & {#2} & {#3} & {#4} \\ 	\FPeval{\res}{clip(#5+#6+#7+#8-1)}
\PP^{\res} & {#5} & {#6} & {#7} & {#8}
\end{array}\right)}

\begin{document}

\title{Rational curves in CICY's in products of two projective spaces}
\author{Filippo F. Favale}
\date{\today}
\address[Filippo F. Favale]{Department of Mathematics, University of Trento, via Sommarive 14,
I-38123 Trento, Italy}
\email{filippo.favale@unitn.it}


\begin{abstract}
Let $X$ be the product of two projective spaces and consider the general CICY threefold $Y$ in $X$ with configuration matrix $A$. We prove the finiteness part of the analogue of the Clemens' conjecture for such a CICY in low bidegrees. More precisely, we prove that the number of smooth rational curves on $Y$ with low bidegree and with nondegenerate birational projection is at most finite (even in cases in which positive dimensional families of degenerate rational curves are known).
\end{abstract}
\maketitle

\section*{Introduction}

The Clemens' conjecture (\cite{Cle84}) states that the generic quintic threefold contains a finite number of smooth rational curves of degree $d$. One of the first successful approach to the proof of this conjecture was done by S. Katz, who proved a stronger version of the conjecture for degree $d\leq 7$, namely
\begin{conj}[\cite{Kat86}]
The scheme of smooth rational curves of degree $d$ on the general quintic threefold is not empty, reduced and finite.
\end{conj}
In \cite{Nij95},\cite{JK96}, following the main idea of its proof,  the authors were able to prove the conjecture in its stronger form for $d\leq 9$. Recently this has been improved: the conjecture is true for $d\leq 11$ (\cite{Cot12}). The approach of Katz is, roughly speaking, as follows. First, one constructs the space $\M_d$ of smooth rational curves in $\PP^4$ of degree $d$ forgetting about the quintic threefold. After that one takes the moduli space $\PP$ of smooth Calabi-Yau threefold in $\PP^4$ and the incidence correspondence $J_d$, i.e., all pairs $(C,Y)\in \M_d\times \PP$ such that $C\subset Y$ with the two projections on $\M_d$ and $\PP$. Then, the conjecture holds for $d$ if $J_d$ is irreducible of dimension equal to the dimension of $\PP$ and if there exists a curve of degree $d$ on the generic quintic threefold. The main ingredient in order to prove the finiteness of the space of curves on the generic quintic threefold are the results about the $6-$regularity (\cite{Mum66}) for curves of low degree proven in \cite{GLP83}.
\vspace{2mm}

\noindent With the same idea in mind, one can ask if the same result is true for other families of Calabi-Yau varieties. For example, consider the families of Calabi-Yau threefolds which are complete intersections in some projective space as done in \cite{JK03}. The work follows again the same idea: one defines an incidence correspondence and answers positively to the finiteness and existence conditions if the degree of the curves is low. The meaning of `low' has to be clarified: one investigates the case where the Calabi-Yau are nondegenerate complete intersections in $\PP^N$ and obtains $5$ families (respectively one in $\PP^4$ - the quintic threefold - two in $\PP^5$, two in $\PP^6$ and only one in $\PP^7$). In \cite{JK03} the authors prove that the generalized Clemens' conjecture holds respectively for $d\leq 7$, $d\leq 6$ and $d \leq 5$ if one considers non degenerate complete intersections in $\PP^5,\PP^6$ and $\PP^7$. Other interesting results on this subject can be found in \cite{Knu12} and \cite{Knu13}.
\vspace{2mm}

\noindent In order to approach the conjecture for a family of Calabi-Yau threefolds in an ambient space $X$, one has two main objectives: the existence of a curve of the right ``degree'' on the general Calabi-Yau threefold $Y$ on the family and the finiteness of such curves in $Y$. 
\vspace{2mm}

\noindent In this paper we investigate the finiteness problem for families of Complete Intersection Calabi-Yau (CICY in short) in the product $X$ of two projective spaces $\PP^{a_1}$ and $\PP^{a_2}$. To approach this case we need a result of regularity similar to the one proven in \cite{GLP83} which is known for curves in the product of two projective spaces (\cite{Loz09}). The Theorem in \cite{Loz09} holds only when $a_i\geq 2$ and only for curves which have a nondegenerate birational projection on the two factors of $X$. For this reason, instead of using the space $\M_{d_1,d_2}$ of smooth rational curves of bidegree $(d_1,d_2)$ in $X$, we use the subspace $\M_{d_1,d_2}'$ of curves which have a nondegenerate birational projection on both the factors and the associated incidence correspondence $J_{d_1,d_2}'$. To each general CICY $Y$ in $X$ we can associate a configuration matrix $A$, a $2$ by $m$ matrix ($m$ is the codimension of $Y$ in $X$) whose columns are the bidegrees of the hypersurfaces that cut $Y$. Our main result is Theorem \ref{THM:BIDEGREES} whose results are the analogue of the finiteness results in \cite{Kat86}, \cite{Nij95}, \cite{JK96}, and \cite{JK03}. It can be stated as follows:
\begin{thm*}
Let $X=\PP^{a_1}\times \PP^{a_2}$ and let $A$ be a matrix configuration of a nondegenerate CICY in $X$. Let us define
\begin{equation}
\begin{array}{l}
W_A= \{(d_1,d_2) \,|\, a_i\leq d_i\} \\
Z_A=\left\lbrace
\begin{array}{c}
(d_1,d_2)\,|\, \forall j=1,\dots, m\quad 
 \exists\, u,v \mbox{ with } u+v=1 \mbox{ such that}\\
a_1\leq d_1\leq a_1+b_{2j}-1+v \quad\mbox{ and }\quad a_2\leq d_2\leq a_2+b_{1j}-1+u
\end{array}
\right\rbrace.
\end{array}
\end{equation}
Then the following hold:
\begin{itemize}
\item the set $Z_A$ is not empty;
\item If $(d_1,d_2)\not\in W_A$, then $\M_{d_1,d_2}'$ is empty;
\item If $(d_1,d_2)\in Z_A$, then $J_{d_1,d_2}'$ is irreducible and the generic CICY in $U_A$ contains at most a finite number of elements of $\M_{d_1,d_2}'$.
\end{itemize}
\end{thm*}

One might wonder if it is possible to work with $\M_{d_1,d_2}$ instead of $\M_{d_1,d_2}'$ in order to obtain the finiteness result but, surprisingly enough, this is not the case. Indeed, for example, in \cite{Som00} is proved that the generic CICY of Tian-Yau type (see Remark \ref{REM:TY}) in $\PP^3\times \PP^3$ contains a positive dimensional familiy of rational curves for bidegrees $(3,3)$, whereas our Theorem says that the curves in $\M_{d_1,d_2}'$ on such general CICY are at most finite. These facts are not incompatible: the family of rational curves constructed in \cite{Som00} is a family of degenerate curves.
\vspace{2mm}

\noindent The main theorem is proved in Section \ref{SEC:MAINTHM}; in Section \ref{SEC:CICY} we briefly fix some notation and discuss how to obtain all the configuration matrices which give CICY's in $X=\PP^{a_1}\times \PP^{a_2}$ which are not equivalent to a family of CICY's in some $\PP^N$ (it is the same as to study complete intersections on $\PP^N$ which are nondegenerate). In Section \ref{SEC:MNONPRIM} we prove Theorem \ref{THM:Pa1Pa211} which gives some results on the case of highly degenerate curves:

\begin{thm*}
Let $X=\PP^{a_1}\times \PP^{a_2}$ and fix a configuration matrix $A$ of a CICY. If $(d_1,d_2)\in \{(0,1),(1,0),(1,1)\}$ then $J_{d_1,d_2}$ is irreducible and the generic Calabi-Yau threefold $Y$ in $X$ with matrix configuration $A$ contains at most a finite number of curves in $\M_{d_1,d_2}$.
\end{thm*}

In the Appendix there is an exhaustive list of configuration matrices for $X=\nobreak\PP^{a_1}\times\nobreak \PP^{a_2}$, which are non degenerate and the list of all the bidegrees for each configuration matrix for which Theorem \ref{THM:BIDEGREES} holds. 
\vspace{2mm}

\noindent This research project was partially supported by FIRB 2012 "Moduli spaces and Applications".

\section{CICY's in $\PP^{a_1}\times \PP^{a_2}$ and configuration matrices}
\label{SEC:CICY}

A {\it Complete Intersection Calabi-Yau threefold} in a variety $X$, {\it CICY} in short, is a complete intersection of hypersurfaces of $X$. We are interested in CICY's in the product of two projective spaces $\PP^{a_1}$ and $\PP^{a_2}$, which will be denoted by $X:=\PP^{a_1}\times \PP^{a_2}$. More generally, if $X$ is the product of several projective spaces, possibly of different dimensions (call $a_i$ with $i=1\dots l$ the dimensions of the projective spaces), and  $Y$ is a CICY in $X$, the codimension of $Y$ in $X$ is given by 
$$m=\sum_{i=1}^{l}a_i-3.$$
The algebraic variety $Y$ it is described by the configuration matrix associated to $Y$, which is a matrix with $m$ columns and $l$ rows specifying the multidegrees of the hypersurfaces that cut $Y$. To be more precise, recall first that 
$$\Pic(X)=\prod_{i=1}^l\pi_i^*\Pic(\PP^{a_i})= \left\langle\left\lbrace\OO_X(\underline{e}_i)\, |\, i=1,\dots,l\right\rbrace\right\rangle,$$
where 
$$\OO_X(\underline{e}_i):=\pi_i^*\OO_{\PP^{a_i}}(i)$$
so that
$$\OO_X(\underline{b})=\OO_X(b_1,\dots,b_l):=\bigotimes_{i=1}^l\OO_X(\underline{e}_i)^{\otimes b_i}.$$
If $Y$ is the zero locus of $s_1,\dots, s_m$ with $s_j\in \H^0(\OO_X(\underline{b}_j))$, then the configuration matrix $A_Y$ associated to $Y$ is simply the matrix $A_Y=\left(b_{ij}\right)$. Often, an extra column on the left is added just to remember the dimension of the factors of $X$:
$$
\left(\begin{array}{c|ccc}
\PP^{a_1} & b_{11} & \cdots & b_{1m} \\
\vdots & \vdots & \ddots & \vdots \\
\PP^{a_l} & b_{l1} & \cdots & b_{lm}
\end{array}\right)
$$
By a simple computation, one sees that if $Y$ is a generic complete intersection of type specified by $A$ in $X=\prod_{i=1}^{l}\PP^{a_i}$ we have 
$$c_1(Y)=\sum_{i=1}^l\left(-a_i-1+\sum_{j=1}^m b_{1j}\right)c_1(\OO_X(\underline{e}_i));$$
so we require the sum of the elements on each row of the matrix to be equal to the dimension of the $\PP^n$ (associated to that row) plus $1$ in order to get a Calabi-Yau. Some configurations cannot give irreducible varieties like, for example,
$$
\left(\begin{array}{c|cc}
\PP^{1} & c & \underline{b}_1 \\
\PP^{a_2} & 0 & \underline{b}_2
\end{array}\right)
$$
which yields a reducible variety for any $c>1$. 
\vspace{2mm}

\noindent It is worth to observe that some configurations are equivalent to others. One can, for example, permute the column or the rows of $A$ but there are some nontrivial ones: namely
$$
\left(\begin{array}{c|cc}
\PP^{a_1} & 1 & \underline{b}_1 \\
\PP^{a_2} & 0 & \underline{b}_2
\end{array}\right)\sim 
\left(\begin{array}{c|c}
\PP^{a_1-1} & \underline{b}_1 \\
\PP^{a_2} & \underline{b}_2
\end{array}\right)
$$
as the column $(1,0)^T$ means a section of $\pi_1^*\OO_{\PP^{a_1}}(1)$ and thus it specifies a hyperplane in $\PP^{a_1}$. The configurations like that on the left are called {\it degenerate}. Several others relations make possible the reduction to simpler configurations. We don't report them but the interested reader can have an insight of them at \cite{GHL13}.
\vspace{2mm}

% \noindent Thanks to these relations, in order to classify CICYs, at least with respect to the bidegree of the hypersurfaces cutting them, it is usefull to introduce a notion of {\it minimality} on the set of admissible matrix configurations. We say that $A$ is {\it minimal} if doesn't exist a smaller configuration matrix $B$ that yields an equivalent familiy of Calabi-Yau threefolds\footnote{E' la definizione che riportano sull'articolo \cite{GHL13}: forse va rivista (quello smaller non \'e che sia troppo preciso...)}.
% \vspace{2mm}

\noindent There is a way to bound the number of CICY's in the product of projective spaces. Denote by $p$ the number of $\PP^1$ in the decomposition of $X$ and by $s$ the numbers of projective spaces of dimension greater than $1$, so that $p+s=l$. If 
$$\alpha=\sum_{i=1}^l (a_i-1)$$
then (see, for example, \cite{GHL13}) any CICY of dimension $3$ is equivalent to one in a suitable $X=\Pi_i^l\PP^{a_i}$ with
\begin{equation}
\label{EQ:BOUND}
p\leq \alpha\leq 6\quad\mbox{ and }\quad s\leq 9.
\end{equation}

\noindent From now on, we assume to be in the case where $X$ is the product of two projective spaces of dimension $a_1$ and $a_2$, respectively. In this case, we have 
$$\alpha=a_1+a_2-2=\dim(X)-2=m+3-2=m+1$$
and the inequalities \ref{EQ:BOUND} gives us
$$p\leq m+1\leq 6.$$
Hence, the codimension of $Y$ can be assumed to be at most $5$. So, an admissible matrix of a CICY can be searched among the $2$ by $m$ matrices with natural entries such that the following hold:
\begin{itemize}
\item $m$ is between $1$ and $5$;
\item $a_1+a_2=3+m$;
\item the sum on the $i-$th row is equal to $a_i+1$;
\item the sum on the $i-$th column is greater than or equal to $2$;
\item if $c\underline{e}_i$ is a column of $A$ then $a_i>1$;
\end{itemize}
In our case, these information alone are sufficient to conclude that the number of such configuration is finite and to write down all of these matrices. For example, in codimension $1$, i.e. for the hypersurfaces case, we only have $2$ configurations, namely
$$\left(\begin{array}{c|c} 
\PP^{1} & {2} \\ 	
\PP^{3} & {4}
\end{array}\right)\quad \CICYu{3}{3}$$
whereas for codimension equal to $2$ we have a total of $11$ configurations, up to symmetries. In codimension $3,4$ and $5$ there are, $22, 14$ and $8$ different configurations respectively. Summing up, we have a total of $57$ configurations for non degenerate CICY in the product of two projective spaces. Nevertheless notice that some of these configurations yield the same family of Calabi-Yau threefolds. This fact is indeed useful (see Remark \ref{REM:ANCHEP1}). For the list of the $57$  configurations and a discussion about the configurations that give the same families, see Appendix \ref{SEC:APPLIST}.

\section{Rational curves in $\PP^{a_1}\times \PP^{a_2}$ of bidegree $(d_1,d_2)$}
\label{SEC:MAINTHM}

From now on, $C$ will be a curve on $X=\PP^{a_1}\times \PP^{a_2}$ with $a_i\geq 1$. Recall that we can define the {\it bidegree} of $C$ to be the pair $\underline{d}=(d_1,d_2)$ if and only if 
$$\deg(\OO_C(\underline{e}_i))=d_i.$$ 
Equivalently, if one defines $H_i$ to be a divisor such that $\OO_X(\underline{e}_i)=\OO_X(H_i)$,
$C$ has bidegree $(d_1,d_2)$ if and only if $C.H_i=d_i$.
\vspace{2mm}

\noindent We will be interested in the moduli space $\M_{\underline{d}}=\M_{d_1,d_2}$ of smooth rational curves of bidegree $(d_1,d_2)$ in $X$. If $C\in \M_{d_1,d_2}$ there exist
$$\alpha=(\alpha_0,\dots,\alpha_{a_1})\in \H^0(\OO_{\PP^1}(d_1))^{\oplus (a_1+1)}$$
$$\beta=(\beta_0,\dots,\beta_{a_2})\in \H^0(\OO_{\PP^1}(d_2))^{\oplus (a_2+1)}$$
such that 
$$
\xymatrix@R=1em{
\PP^1\ar[r] & C \ar@{^{(}->}[r]& \PP^{a_1}\times \PP^{a_2} \\
(t:u)\ar@{|->}[rr] & & (\alpha(t:u),\beta(t:u)) 
}$$
The choice of $\alpha$ and $\beta$ it is not unique: if we fix projective coordinates on $\PP^1$ they are defined up to multiplication by scalars so we have an element on 
$$\PP(\H^0(\OO_{\PP^1}(d_1))^{a_1+1})\times \PP(\H^0(\OO_{\PP^1}(d_2))^{a_2+1}).$$

In order to get rid of the choice of the coordinates on $\PP^1$ we can consider the natural action of $\PGL(2)$. We can identify $\M_{d_1,d_2}$ with an open and irreducible set in 
\begin{equation}
\left(\PP(\H^0(\OO_{\PP^1}(d_2))^{a_2+1})\times \PP(\H^0(\OO_{\PP^1}(d_2))^{a_2+1})\right)/\PGL(2).
\end{equation}
For example, in order to obtain really a curve of bidegree $(d_1,d_2)$  we need to discard all the maps $(\alpha,\beta)$ such that $\alpha\in (<f>\otimes H^0(\OO_{\PP^1}(d_1')))$ for some $f\in H^0(\OO_{\PP^1}(d_1''))$ with $d_1=d_1'+d_1''$ (and the same for $\beta$). This is indeed a closed set in $\PP(\H^0(\OO_{\PP^1}(d_1))^{a_1+1})\times \PP(\H^0(\OO_{\PP^1}(d_2))^{a_2+1})$ like the set that yield singular curves. In particular, we have that
\begin{multline}
\label{EQ:DIMM}
\dim(\M_{d_1,d_2})=((d_1+1)(a_1+1)-1)+((d_2+1)(a_2+1)-1)-3=\\
=d_1(a_1+1)+d_2(a_2+1)+(a_1+a_2-3)=d_1(a_1+1)+d_2(a_2+1)+m.
\end{multline}

\noindent Define $\M_{d_1,d_2}'$ to be the moduli space of smooth rational curves in $X=\PP^{a_1}\times \PP^{a_2}$ of bidegree $(d_1,d_2)$ with nondegenerate birational projections on the factors of $X$. The elements of $\M_{d_1,d_2}'$ are then curves in $\M_{d_1,d_2}$ such that $\pi_i|_{C}:C\rightarrow C_i:=\pi_i(C)\subset \PP^a_1$ is a birational morphism on the image and $C_i$ is a nondegenerate curve in $\PP^a_i$. The last condition is equivalent to require that $C$ does not lie in a $\PP^{a_1-1}\times \PP^{a_2}$ or in a $\PP^{a_1}\times \PP^{a_2-1}$ inside $X$. If $C\subset \PP^{a_1-1}\times \PP^{a_2}$ then, for each $(\alpha,\beta)$ whose image is $C$, we have $\alpha\in H^0(\OO_{\PP^1}(d_1))^{a_1+1}$ with $\{\alpha_j\}_{j=1}^{a_1+1}$ not indipendent, which is a closed condition if $a_1\leq d_1$. Conversely, if $\alpha$ and $\beta$ are two collections of indipendent homogenous forms, they live on an open set. This is enough to prove that having nondegenerate projections is an open condition. Asking $\pi_i|_C$ to be birational onto its image is also an open condition by semicontinuity so we have that $\M_{d_1,d_2}'$ is an open and irreducible set in $\M_{d_1,d_2}$.

\begin{lem}
\label{LEM:H0ICX} 
Let $C$ be a curve of bidegree $(d_1,d_2)$ on $X=\PP^{a_1}\times \PP^{a_2}$ and let $\LL=\OO_X(b_1,b_2)$ with $b_1,b_2\geq 0$. Then
$$\h^0(I_{C/X}(b_1,b_2))=\binom{a_1+b_1}{b_1}\binom{a_2+b_2}{b_2}-1-b_1d_1-b_2d_2+\h^1(\I_{C/X}(b_1,b_2))$$
\end{lem}

\begin{proof}
Start from the exact sequence of sheaves that defines $\OO_C$ and twist it by $\LL$ obtaining
$$0\rightarrow \I_{C/X}(b_1,b_2)\rightarrow \OO_X(b_1,b_2)\rightarrow \OO_C(b_1,b_2)\rightarrow 0$$
and its associated cohomology sequence
\begin{equation}
\renewcommand{\arraystretch}{1.35}
\begin{array}{l}
0\rightarrow \H^0(\I_{C/X}(b_1,b_2))\rightarrow \H^0(\OO_X(b_1,b_2))\rightarrow \H^0(\OO_C(b_1,b_2))\rightarrow \\
\phantom{0} \rightarrow \H^1(\I_{C/X}(b_1,b_2))\rightarrow \H^1(\OO_X(b_1,b_2)) 
\end{array}
\end{equation}
By K\"unnet formula we have 
$$\H^1(\OO_X(b_1,b_2))=\left(\H^0(\OO_{\PP^{a_1}}(b_1))\otimes \H^1(\OO_{\PP^{a_2}}(b_2))\right)\oplus \left(\H^1(\OO_{\PP^{a_1}}(b_1))\otimes \H^0(\OO_{\PP^{a_2}}(b_2))\right)$$
so $\H^1(\OO_X(b_1,b_2))=0$ because $\H^1(\OO_{\PP^{a_i}}(b_i))=0$ for $b_i\geq 0$.
The claim follows by observing that $\OO_C(b_1,b_2)$ has degree $b_1d_1+b_2d_2$ as a line bundle on the rational curve $C$ because, by hypotesis, $\OO_C(\underline{e}_i)$ has degree $d_i$.
\end{proof}

% --------------------------------------------------------------
% 
% --------------------------------------------------------------

\section{CICY's in $\PP^{a_1}\times \PP^{a_2}$ and rational curves in them}

Consider the variety $X=\PP^{a_1}\times \PP^{a_2}$ and let $Y$ be a Calabi-Yau threefold in $X$ with matrix configuration $A_Y=A$. More precisely, if $Y$ is the zero locus of the sections $$s_j\in \H^0(\OO_X(b_{1j},b_{2j})),$$ where $j=1,\dots,m$. Then $A=(b_{ij})$ and we have $\sum_j b_{ij}=a_i+1$ to ensure that $c_1(Y)=0$.
Set
$$\PS_A:=\prod_{j=1}^m \PP(\H^0(\OO_X(b_{1j},b_{2j}))),$$
and denote by $\PP_A$ the moduli space of CICY's with matrix configuration $A$ in $X$. Let $U_A$ be the open set of $\PS_A$ that correspond to the choices of set of sections that yield smooth complete intersections in $X$. Thus, we have a map
$$\sigma:U_A\rightarrow \PP_A$$
which is a surjective morphism sending $\underline{s}$ to its zero locus $V(s_1,\dots, s_m)$.

The dimension of $U_A$ is 
\begin{equation}
\dim(U_A)=\dim(\PS_A)=\sum_{j=1}^m\left(\binom{a_1+b_{1j}}{a_1}\binom{a_2+b_{2j}}{a_2}-1\right).
\end{equation}

Denote by $J_{d_1,d_2}$ the incidence correspondence given by
$$J_{d_1,d_2}:=\{(C,\underline{s})\in \M_{d_1,d_2}\times U_A \,|\, C\subset V(s_1,\dots,s_m)=\sigma(\underline{s})\}$$
and by $p$ and $q$ the canonical projection on $\M_{d_1,d_2}$ and $U_A$ respectively. We will denote by $J_{d_1,d_2}'$ the preimage of $\M_{d_1,d_2}'$ under $p$. 

$$\xymatrix{
J_{d_1,d_2}'\ar@/^1pc/[rrd]^{q'}\ar[dd]_{p'}\ar@{^{(}->}[dr] \\
& J_{d_1,d_2} \ar[r]^q \ar[d]_{p} & U_A \\
\M_{d_1,d_2}'\ar@{^{(}->}[r] & \M_{d_1,d_2}
}$$

Set theoretically, the fiber of $p$ over $C$ is precisely the set of the pairs $(C,Y)$, where $Y=\sigma(\underline{s})$ is a Calabi-Yau threefold with matrix configuration $A$ containing $C$. This means that $s_j$ has to be in the ideal of $C$ in $X$ for all $j$. More precisely, we have $s_j\in \PP(\H^0(\I_{C/X}(b_{1j},b_{2j})))$ so that
\begin{equation}
\label{EQ:FIBPROJ}
p^{-1}(C)=\prod_{j=1}^m\PP(\H^0(\I_{C/X}(b_{1j},b_{2j}))).
\end{equation}

\begin{lem}
\label{EQ:DIMFIB2}
If $C\in \M_{d_1,d_2}$, the dimension of the fiber of $p$ over $C$ is given by 
$$\dim(p^{-1}(C))=\dim(U_A)-\dim(\M_{d_1,d_2})+\sum_{j=1}^m\h^1(\I_{C/X}(b_{1j},b_{2j})).$$
\end{lem}

\begin{proof}
By Equation \ref{EQ:FIBPROJ} we have
$$\dim(p^{-1}(C))=\sum_{j=1}^m\dim\left(\PP(\H^0(\I_{C/X}(b_{1j},b_{2j})))\right).$$
By Lemma \ref{LEM:H0ICX} we have
\begin{multline}
\dim(p^{-1}(C))=\\
\sum_{j=1}^m\left(\binom{a_1+b_{1j}}{b_{1j}}
=\binom{a_2+b_{2j}}{b_{2j}}-(1+b_{1j}d_1+b_{2j}d_2)+\h^1(\I_{C/X}(b_{1j},b_{2j}))-1\right)=\\
=\sum_{j=1}^m\left(\binom{a_1+b_{1j}}{b_{1j}}
\binom{a_2+b_{2j}}{b_{2j}}-1\right)+\sum_{j=1}^m(1+b_{1j}d_1+b_{2j}d_2)+\sum_{j=1}^m\h^1(\I_{C/X}(b_{1j},b_{2j}))=\\
=\dim(U_A)-(m+d_1(a_1+1)+d_2(a_2+1))+\sum_{j=1}^m\h^1(\I_{C/X}(b_{1j},b_{2j}))=\\
=\dim(U_A)-\dim(\M_{d_1,d_2})+\sum_{j=1}^m\h^1(\I_{C/X}(b_{1j},b_{2j}))
\end{multline}
as claimed.
\end{proof}

\begin{prop}
\label{PROP:Pa1Pa2}
Let $A=(b_{ij})$ be a matrix configuration that yields a CICY in $X=\PP^{a_1}\times \PP^{a_2}$. Assume that every curve in $\M_{d_1,d_2}$  (resp. every curve in $\M_{d_1,d_2}'$) is such that $\h^1(\I_{C/X}(b_{1j},b_{2j}))=0$ for all $j$. Then
\begin{itemize}
\item $J_{d_1,d_2}$ is irreducible (resp. $J_{d_1,d_2}'$);
\item for the generic Calabi-Yau threefold in $X$ with matrix configuration $A$, the set of curves (resp. the set of curves with nondegenerate birational projections) of bidegree $(d_1,d_2)$ contained in $Y$ is either finite or empty.
\end{itemize}
\end{prop}

\begin{proof}
Under the assumptions and by Lemma \ref{EQ:DIMFIB2}, we have that the fibers of $p$ have constant dimension given by 
$$\dim(p^{-1}(C))=\dim(U_A)-\dim(\M_{d_1,d_2}).$$
Hence, we deduce that  
$$\dim(J_{d_1,d_2})=\dim(\M_{d_1,d_2})+\dim(p^{-1}(C))=\dim(U_A).$$
Then, as in \cite{Kat86}[Lemma 1.4], we can conclude that $J_{d_1,d_2}$ is irreducible of dimension equal to $\dim(U_A)$. 
Since $q:J_{d_1,d_2}\rightarrow U_A$ is a morphism between two varieties with the same dimension, it is either dominant (and in this case its generic fiber has dimension $0$) or it has image contained in a proper subspace of $\PP$ (and in this case its generic fiber is empty). To conclude, it is enough to observe that the fiber over $Y\in U_A$ is exactly the set of rational curves of bidegree $(d_1,d_2)$ contained in $Y$. The case with $\M_{d_1,d_2}'$ is analogous.
\end{proof}

Given a configuration matrix $A=(b_{ij})$ for a CICY $Y$ in $X=\PP^{a_1}\times \PP^{a_2}$ (as usual, $m$ denotes the codimension of $Y$ in $X$ and is the number of columns of $A$), we will denote by $W_A$ and $Z_A$ the sets defined by 
\begin{equation}
\begin{array}{l}
W_A= \{(d_1,d_2) \,|\, a_i\leq d_i\} \\
Z_A=\left\lbrace
\begin{array}{c}
(d_1,d_2)\,|\, \forall j=1,\dots, m\quad 
 \exists u,v \mbox{ with } u+v=1 \mbox{ such that}\\
a_1\leq d_1\leq a_1+b_{2j}-1+v \quad\mbox{ and }\quad a_2\leq d_2\leq a_2+b_{1j}-1+u
\end{array}
\right\rbrace.
\end{array}
\end{equation}

\begin{thm}
\label{THM:BIDEGREES}
Let $X$ be $\PP^{a_1}\times \PP^{a_2}$ with $a_i\geq 2$ and consider a nondegenerate matrix configuration\footnote{Hence, it suffices to assume $a_1+a_2\leq 8$ and consider only the matrices in Appendix \ref{SEC:APPLIST}.} $A$ for a CICY in $X$. Then the following hold:
\begin{itemize}
\item the set $Z_A$ is not empty;
\item if $(d_1,d_2)\not\in W_A$, then $\M_{d_1,d_2}'$ is empty;
\item if $(d_1,d_2)\in Z_A$, then $J_{d_1,d_2}'$ is irreducible and the generic CICY in $U_A$ contains at most a finite number of elements of $\M_{d_1,d_2}'$.
\end{itemize}
\end{thm}

\begin{proof}
To see that $Z_A$ is not empty, it is enough to observe that 
$(a_1,a_2)\in Z_A$ for all $A$ because $b_{ij}\geq 0$ and $b_{1j}$ and $b_{2j}$ cannot be both equal to $0$. The second claim is also easy: we have already observed that if either $a_1>d_1$ or $a_2>d_2$ then every curve of bidegree $(d_1,d_2)$ has to be degenerate: $\M_{d_1,d_2}'$ is empty.
In what follows, we may assume $a_i\leq d_i$ to ensure $\M_{d_1,d_2}'$ is not empty. The proof of the last claim follows from the fact that $(d_1,d_2)\in Z_A$ then 
\begin{equation}
\label{EQ:LOZ}
H^1(\I_{C/X}(b_{1j},b_{2j}))=0
\end{equation}
for all $j=1,\dots,m$: we can indeed conclude by Proposition \ref{PROP:Pa1Pa2}. To prove this, take $C\in \M_{d_1,d_2}'$ which is not empty by assumption. 
Consider a column $(b_{1j},b_{2j})$ of the configuration matrix. By \cite{Loz09} we have that $C$ is $(d_2-a_2+1,d_1-a_1+1)-$regular. Hence for all $u,v$ such that $u+v=1$ we have
$$H^1(\I_{C/X}(d_2-a_2+1-u,d_v-a_v+1-v))=0.$$
By \cite{MS04} we have also that $$H^1(\I_{C/X}(d_2-a_2+1-u+n_1,d_v-a_v+1-v+n_2))=0$$
for $(n_1,n_2)\in \N^2$ so if we ask
$$b_{1j}\geq d_2-a_2+1-u, b_{2j}\geq d_1-a_1+1-v$$
we have the wanted result: these inequalities, together with $a_i\leq d_i$, are exactly the ones defining $Z_A$.
\end{proof}

\begin{remark}
There are some cases for which $Z_A$ is reduced to the unique pair $(a_1,a_2)$, as well as cases for which $Z_A$ has more elements. The complete list of the pairs associated with the relative configuration matrix can be found in Appendix \ref{APP:TABLE}.
\end{remark}

\begin{remark}
\label{REM:ANCHEP1}
Note that Theorem \ref{THM:BIDEGREES} is only applicable for products $\PP^{a_1}\times \PP^{a_2}$ with $a_i\geq 2$. Nevertheless, because more configuration matrices represent the same family of CICY, one may be able to say something also in cases for which $a_1=1$. For example, the relation
$$\left(\begin{array}{c|c}
\PP^{1} & 2\underline{c}_1\\
\PP^{a_2} & \underline{c}_2
\end{array}\right)\sim \left(\begin{array}{c|cc}
\PP^{2} & 2 & \underline{c}_1 \\
\PP^{a_2} & 0 & \underline{c}_2 
\end{array}\right)$$
works well in this direction. See Appendix \ref{SEC:APPLIST} for major details.
\end{remark}

\begin{remark}
\label{REM:TY}
Here we focus on the case of non degenerate rational curves and one might wonder if the same result holds without this assumption, but this is not the case. Consider, for example, the generic CICY $Y$ in $\PP^3\times \PP^3$ with matrix configuration
$$A=\CICYt{3}{1}{0}{0}{1}{3},$$
the so called Tian-Yau CICY.
By Theorem \ref{THM:Pa1Pa211} we have that $Y$ contains at most a finite number of nondegenerate rational curves of bidegree $(3,3)$ but it can be shown (see \cite{Som00}) that $Y$ contains a positive dimensional family of degenerate curves of bidegree $(3,3)$.
\end{remark}

% --------------------------------------------------------
% --------------------------------------------------------
% --------------------------------------------------------

\section{Results on $\M_{d_1,d_2}$ for $d_i\leq 1$ for a CICY in $\PP^{a_1}\times \PP^{a_2}$}
\label{SEC:MNONPRIM}

In this section $X$ will denote $\PP^{a_1}\times \PP^{a_2}$. We will prove the following:

\begin{thm}
\label{THM:Pa1Pa211}
Let $X=\PP^{a_1}\times \PP^{a_2}$ and fix a configuration matrix $A$ of a CICY. If $(d_1,d_2)\in \{(0,1),(1,0),(1,1)\}$, then $J_{d_1,d_2}$ is irreducible and the generic Calabi-Yau threefold $Y$ in $X$ with matrix configuration $A$ contains at most a finite number curves of $\M_{d_1,d_2}$.
\end{thm}

We start with the following lemma
\begin{lem}
\label{LEM:DEG}
Let $C$ be a smooth subvariety of $X=\PP^{a_1}\times \PP^{a_2}$ and assume that there exists a hyperplane $H_1$ of $\PP^{a_1}$ such that $C\subset H_1\times \PP^{a_2}:=H$. Then $$H^1(\I_{C/X}(b_1,b_2))\simeq H^1(\I_{C/H}(b_1,b_2))$$
for any $b_1,b_2\geq 0$.
\end{lem}

\begin{proof}
Since $H_1$ is a hyperplane in $\PP^{a_1}$, we have $\I_{H/X}=\OO_X(-1,0)$. From 
$C\subset H\subset X$
we have the exact sequence
$$0\rightarrow\I_{H/X}\rightarrow\I_{C/X}\rightarrow \I_{C/H}\rightarrow 0,$$
which we twist with $\OO_X(b_1,b_2)$. We can conclude by observing that, for $p=1,2$, we have
$$H^p(\I_{H/X}(b_1,b_2))\simeq H^p(\OO_{X}(b_1-1,b_2))=0$$
as $b_1-1,b_2\geq -1$.
\end{proof}

\begin{proof}(of Theorem \ref{THM:Pa1Pa211})
We cannot use \cite{Loz09} because every curve $C$ in $\M_{d_1,d_2}$ has a degenerate birational projection. We will use instead the fact that $C$ is indeed very degenerate by showing "directly" that $H^1(\I_C/X(b_{1j},b_{2j}))=0$ for each column $(b_{1j},b_{2j})$ of $A$ and then concluding as in Proposition \ref{PROP:Pa1Pa2}. 
Consider the case $(d_1,d_2)=(1,1)$ (the other are similar). Let $C$ be a curve of bidegree $(1,1)$ in $X$. Then $C$ is the image of $\gamma: (t_0:t_1)\rightarrow (\alpha_1,\alpha_2)$ where $\alpha_i\in \H^0(\OO_{\PP^1}(1))^{a_i+1}$. If $\alpha_i=(\alpha_{i0},\dots ,\alpha_{ia_i})$ then it is clear that there exist $a_i-1$ indipendent linear relations between them (if there are $a_i$ linear indipendent relations, then the image of $\pi_i\circ\gamma$ is a point and thus $C$ cannot have bidegree $(1,1)$). In particular, there exist $a_i-1$ hyperplanes $H_{ik_i}$ ($k_i=1..(a_i-1)$) in $\PP^{a_i}$ such that $C$ is contained in $H_{1k_1}\times \PP^{a_2}$ and in $\PP^{a_1}\times H_{2k_2}$. By intersecting all these hypersurfaces we obtain a $\PP^1\times \PP^1$.
Using Lemma \ref{LEM:DEG} we obtain
$$\H^1(\I_{C/X}(b_{1j},b_{2j}))\simeq \H^1(\I_{C/\PP^1\times \PP^1}(b_{1j},b_{2j}))$$
for each $j=1,\dots,m$. The thesis follows immediately because $C$ is a divisor of bidegree $(1,1)$ on $\PP^1\times \PP^1$:
$$\H^1(\I_{C/\PP^1\times \PP^1}(b_{1j},b_{2j}))=\H^1(\OO_{\PP^1\times \PP^1}(b_{1j}-1,b_{2j}-1))=0.$$
\end{proof}

\begin{remark}
Both Lemma \ref{LEM:DEG} and Theorem \ref{THM:Pa1Pa211} can be stated for arbitrary products of projective spaces. Indeed, both proofs can be adapted easily to this case. Hence the generic (non degenerate) CICY in any product of projective spaces, without assumption on the type of its configuration matrix, contains at most a finite number of smooth rational curves of multidegrees $(d_1,\dots, d_r)$ with $d_i\leq 1$ (not all $0$).
\end{remark}

% --------------------------------------------------------
% --------------------------------------------------------
% --------------------------------------------------------

\appendix

\section{List of CICY's in $\PP^{a_1}\times \PP^{a_2}$}
\label{SEC:APPLIST}

Let $Y$ be a CICY in $X=\PP^{a_1}\times \PP^{a_2}$ and call $A$ its configuration matrix. Then either $Y$ is degenerate or it has a matrix configuration equivalent to one of the following $57$ matrices. The matrix configurations are divided by the numbers of columns (i.e., the codimension of the associated CICY in $X$) and are classified up to symmetries (and excluding the degenerate cases). Keep in mind that also with these restrictions, more configurations may yield the same family of CICY's. Indeed, it is always true that
$$\left(\begin{array}{c|c}
\PP^{1} & 2\underline{c}_1\\
\PP^{a_2} & \underline{c}_2
\end{array}\right)\sim \left(\begin{array}{c|cc}
\PP^{2} & 2 & \underline{c}_1 \\
\PP^{a_2} & 0 & \underline{c}_2 
\end{array}\right)$$
which tells us, for example, that
$$\left(\begin{array}{c|c}
\PP^{1} & {2}\\
\PP^{3} & {4}
\end{array}\right)\sim \left(\begin{array}{c|cc}
\PP^{2} & 2 & 1 \\
\PP^{3} & 0 & 4 
\end{array}\right).$$
The reason to treat them as separate cases is that the methods used in the main theorem may apply only for one of the descriptions (the example just presented is one of these cases). 
The matrices with a superscript are those, up to the author's knowledge, for which this happens (more precisely, in the following list, two matrices with the same superscript represent the same family of Calabi-Yau threefold).\vspace{2mm}

\noindent At last, the matrices with a $\bigstar$ are those for which the corresponding ambient space is a $\PP^1\times \PP^{a_2}$.
\vspace{2mm}

Codimension $1$: $2$ configurations.
$$\CICYu{2}{4}^{\mkern-9mu(I)}\quad \CICYu{3}{3}$$

Codimension $2$: $11$ configurations.
$$
\CICYd{1}{1}{1}{4}\quad
\CICYd{1}{1}{2}{3}\quad 
\CICYd{2}{0}{1}{4}^{\mkern-9mu(II)\mkern-18mu}\quad
\CICYd{2}{0}{2}{3}^{\mkern-9mu(III)\mkern-27mu}
$$
$$
\CICYd{2}{0}{3}{2}^{\mkern-9mu(IV)\mkern-18mu}\quad 
\CICYd{2}{1}{0}{4}^{\mkern-9mu(I)\mkern-9mu}\quad 
\CICYd{2}{1}{1}{3}\quad 
\CICYd{2}{1}{2}{2}
$$ 
$$\CICYd{2}{1}{3}{1}\quad \CICYd{3}{0}{1}{3}\quad \CICYd{3}{0}{2}{2}\quad$$

Codimension $3$: $22$ configurations.
$$\CICYt{1}{1}{0}{1}{1}{4}\quad \CICYt{1}{1}{0}{1}{2}{3}\quad
\CICYt{1}{1}{0}{1}{3}{2}\quad \CICYt{1}{1}{0}{2}{2}{2}$$
$$
\CICYt{2}{0}{0}{1}{2}{3}^{\mkern-9mu(V)\mkern-9mu}\quad
\CICYt{2}{0}{0}{2}{2}{2}^{\mkern-9mu(VI)\mkern-18mu}\quad 
\CICYt{2}{1}{0}{0}{1}{4}^{\mkern-9mu(II)\mkern-18mu}\quad 
\CICYt{2}{1}{0}{0}{2}{3}^{\mkern-9mu(III)\mkern-27mu}
$$
$$
\CICYt{2}{1}{0}{0}{3}{2}^{\mkern-9mu(IV)\mkern-18mu}\quad
\CICYt{1}{1}{1}{1}{3}{1}\quad
\CICYt{1}{1}{1}{2}{1}{2}\quad 
\CICYt{2}{1}{0}{1}{1}{3}
$$
$$\CICYt{2}{1}{0}{1}{2}{2}\quad \CICYt{2}{1}{0}{2}{1}{2}\quad
\CICYt{2}{1}{1}{0}{1}{3}\quad \CICYt{2}{1}{1}{1}{1}{2}$$
$$\CICYt{2}{1}{1}{2}{1}{1}\quad \CICYt{2}{2}{0}{1}{0}{3}\quad
\CICYt{2}{2}{0}{1}{1}{2}\quad \CICYt{2}{2}{0}{2}{0}{2}$$
$$\CICYt{3}{0}{0}{1}{2}{2}\quad \CICYt{3}{1}{0}{0}{1}{3}$$

Codimension $4$: $14$ configurations.
$$
\CICYq{1}{1}{0}{0}{1}{1}{2}{3}\quad
\CICYq{1}{1}{0}{0}{1}{2}{2}{2}\quad
\CICYq{2}{0}{0}{0}{1}{2}{2}{2}^{\mkern-9mu(VII)\mkern-27mu}\quad 
\CICYq{2}{1}{0}{0}{0}{1}{2}{3}^{\mkern-9mu(V)\mkern-9mu}
$$
$$
\CICYq{2}{1}{0}{0}{0}{2}{2}{2}^{\mkern-9mu(VI)\mkern-18mu} \quad 
\CICYq{2}{1}{0}{0}{1}{1}{2}{2} \quad 
\CICYq{1}{1}{1}{0}{1}{1}{2}{2} \quad
\CICYq{1}{1}{1}{0}{1}{1}{1}{3} 
$$
$$
\CICYq{1}{1}{1}{1}{2}{1}{1}{1}\quad 
\CICYq{2}{1}{1}{0}{0}{1}{1}{3}\quad
\CICYq{2}{1}{1}{0}{0}{1}{2}{2}\quad 
\CICYq{2}{1}{1}{0}{1}{1}{1}{2}
$$
$$\CICYq{2}{2}{0}{0}{0}{1}{2}{2}\quad \CICYq{3}{1}{0}{0}{0}{1}{2}{2}$$

Codimension $5$: $8$ configurations.
$$\mkern-9mu\bigstar\mkern-9mu\left(\begin{array}{c|ccccc}
\PP^{1} & {1} & {1} & {0} & {0} & {0}\\
\PP^{7} & {1} & {1} & {2} & {2} & {2}
\end{array}\right)\quad \left(\begin{array}{c|ccccc}
\PP^{2} & {2} & {1} & {0} & {0} & {0}\\
\PP^{6} & {0} & {1} & {2} & {2} & {2}
\end{array}\right)^{\mkern-9mu(VII)\mkern-27mu}\quad \left(\begin{array}{c|ccccc}
\PP^{3} & {1} & {1} & {1} & {1} & {0}\\
\PP^{5} & {1} & {1} & {1} & {1} & {2}
\end{array}\right)$$ $$\left(\begin{array}{c|ccccc}
\PP^{4} & {1} & {1} & {1} & {1} & {1}\\
\PP^{4} & {1} & {1} & {1} & {1} & {1}
\end{array}\right)\quad 
\left(\begin{array}{c|ccccc}
\PP^{2} & {1} & {1} & {1} & {0} & {0}\\
\PP^{6} & {1} & {1} & {1} & {2} & {2}
\end{array}\right)
\quad \left(\begin{array}{c|ccccc}
\PP^{3} & {2} & {1} & {1} & {0} & {0}\\
\PP^{5} & {0} & {1} & {1} & {2} & {2}
\end{array}\right)$$ $$\left(\begin{array}{c|ccccc}
\PP^{4} & {2} & {1} & {1} & {1} & {0}\\
\PP^{4} & {0} & {1} & {1} & {1} & {2}
\end{array}\right)\quad \left(\begin{array}{c|ccccc}
\PP^{4} & {2} & {2} & {1} & {0} & {0}\\
\PP^{4} & {0} & {0} & {1} & {2} & {2}
\end{array}\right)$$

\section{Bidegrees for which Theorem \ref{THM:BIDEGREES} holds}
\label{APP:TABLE}
In Theorem \ref{THM:BIDEGREES} we concluded that the set of rational curves of $\M_{d_1,d_2}'$ inside the generic CICY with matrix configuration $A$ in $\PP^{a_1}\times \PP^{a_2}$ is either finite or empty if $(d_1,d_2)\in Z_A$. We proved that $Z_A$ is not empty; Still it is worth checking for which values the Theorem holds. The following table displays $Z_A$ for all the configuration matrices $A$ with $a_1,a_2\geq 2$, such that $|Z_A|\geq 2$ and that have at least $2$ columns. 
\begin{small}
$$
\mkern-40mu\begin{array}{|c|c|c|c|}
\hline
m & X & A & Z_A \\
\hline\hline
2 &
\PP^{2}\times \PP^{3} &
\left(\begin{array}{cc}
2 & 1 \\
0 & 4
\end{array}\right) &
( 2, 3 ),
( 2, 4 )
\\ \hline
2 &
\PP^{2}\times \PP^{3} &
\left(\begin{array}{cc}
2 & 1 \\
1 & 3
\end{array}\right) &
\begin{array}{c}
(2,3),(2,4) \\
(3,3),(3,4)
\end{array}
\\ \hline
2 &
\PP^{2}\times \PP^{3} &
\left(\begin{array}{cc}
2 & 1 \\
2 & 2
\end{array}\right) &
\begin{array}{c}
(2,3), (2,4) \\ 
(3,4), (4,3) \\
(3,3)
\end{array}
\\ \hline
2 &
\PP^{2}\times \PP^{3} &
\left(\begin{array}{cc}
2 & 1 \\
3 & 1
\end{array}\right) &
\begin{array}{c}
( 2, 3 ),( 2, 4 ) \\
( 3, 3 )
\end{array}
\\ \hline
2 &
\PP^{2}\times \PP^{3} &
\left(\begin{array}{cc}
3 & 0 \\
1 & 3
\end{array}\right) &
( 2, 3 ),
( 3, 3 )
\\ \hline
2 &
\PP^{2}\times \PP^{3} &
\left(\begin{array}{cc}
3 & 0 \\
2 & 2
\end{array}\right) &
( 2, 3 ),
( 3, 3 )
\\ \hline\hline
3 &
\PP^{2}\times \PP^{4} &
\left(\begin{array}{ccc}
1 & 1 & 1 \\
1 & 3 & 1
\end{array}\right) &
\begin{array}{c}
( 2, 4 ), ( 2, 5 ) \\
( 3, 4 )
\end{array}
\\ \hline
3 &
\PP^{2}\times \PP^{4} &
\left(\begin{array}{ccc}
1 & 1 & 1 \\
2 & 1 & 2
\end{array}\right) &
\begin{array}{c}
( 2, 4 ), ( 2, 5 ) \\
( 3, 4 )
\end{array}
\\ \hline
3 &
\PP^{2}\times \PP^{4} &
\left(\begin{array}{ccc}
2 & 1 & 0 \\
1 & 1 & 3
\end{array}\right) &
( 2, 4 ),
( 3, 4 )
\\ \hline
3 &
\PP^{2}\times \PP^{4} &
\left(\begin{array}{ccc}
2 & 1 & 0 \\
1 & 2 & 2
\end{array}\right) &
( 2, 4 ),
( 3, 4 )
\\ \hline
3 &
\PP^{2}\times \PP^{4} &
\left(\begin{array}{ccc}
2 & 1 & 0 \\
2 & 1 & 2
\end{array}\right) &
( 2, 4 ),
( 3, 4 )
\\ \hline
3 &
\PP^{2}\times \PP^{4} &
\left(\begin{array}{ccc}
3 & 0 & 0 \\
1 & 2 & 2
\end{array}\right) &
( 2, 4 ),
( 3, 4 )
\\ \hline
\end{array}
\quad %% END OF FIRST COLUMN
\begin{array}{|c|c|c|c|}
\hline
m & X & A & Z_A \\
\hline\hline
3 &
\PP^{3}\times \PP^{3} &
\left(\begin{array}{ccc}
2 & 1 & 1 \\
1 & 1 & 2
\end{array}\right) &
\begin{array}{c}
( 3, 3 ), ( 3, 4 ) \\
( 4, 3 )
\end{array}
\\ \hline
3 &
\PP^{3}\times \PP^{3} &
\left(\begin{array}{ccc}
2 & 1 & 1 \\
2 & 1 & 1
\end{array}\right) &
\begin{array}{c}
( 3, 3 ), ( 3, 4 ) \\
( 4, 3 )
\end{array}
\\ \hline
3 &
\PP^{3}\times \PP^{3} &
\left(\begin{array}{ccc}
2 & 2 & 0 \\
1 & 1 & 2
\end{array}\right) &
( 3, 3 ),
( 4, 3 )
\\ \hline
3 &
\PP^{3}\times \PP^{3} &
\left(\begin{array}{ccc}
2 & 1 & 1 \\
0 & 1 & 3
\end{array}\right) &
( 3, 3 ),
( 3, 4 )
\\ \hline\hline
4 &
\PP^{2}\times \PP^{5} &
\left(\begin{array}{cccc}
1 & 1 & 1 & 0 \\
1 & 1 & 1 & 3
\end{array}\right) &
( 2, 5 ),
( 3, 5 )
\\ \hline
4 &
\PP^{2}\times \PP^{5} &
\left(\begin{array}{cccc}
1 & 1 & 1 & 0 \\
1 & 1 & 2 & 2
\end{array}\right) &
( 2, 5 ),
( 3, 5 )
\\ \hline
4 &
\PP^{2}\times \PP^{5} &
\left(\begin{array}{cccc}
2 & 1 & 0 & 0 \\
1 & 1 & 2 & 2
\end{array}\right) &
( 2, 5 ),
( 3, 5 )
\\ \hline
4 &
\PP^{3}\times \PP^{4} &
\left(\begin{array}{cccc}
1 & 1 & 1 & 1 \\
2 & 1 & 1 & 1
\end{array}\right) &
\begin{array}{c}
( 3, 4 ), ( 4, 5 ) \\
( 4, 4 )
\end{array}
\\ \hline
4 &
\PP^{3}\times \PP^{4} &
\left(\begin{array}{cccc}
2 & 1 & 1 & 0 \\
1 & 1 & 1 & 2
\end{array}\right) &
( 3, 4 ),
( 4, 4 )
\\ \hline\hline
5 &
\PP^{2}\times \PP^{6} &
\left(\begin{array}{ccccc}
1 & 1 & 1 & 0 & 0 \\
1 & 1 & 1 & 2 & 2
\end{array}\right) &
( 2, 6 ),
( 3, 6 )
\\ \hline
5 &
\PP^{3}\times \PP^{5} &
\left(\begin{array}{ccccc}
1 & 1 & 1 & 1 & 0 \\
1 & 1 & 1 & 1 & 2
\end{array}\right) &
( 3, 5 ),
( 4, 5 )
\\ \hline
5 &
\PP^{4}\times \PP^{4} &
\left(\begin{array}{ccccc}
1 & 1 & 1 & 1 & 1 \\
1 & 1 & 1 & 1 & 1
\end{array}\right) &
\begin{array}{c}
( 4, 4 ), ( 4, 5 ) \\
( 5, 4 )
\end{array}
\\ \hline
\end{array}
$$

\end{small}

\noindent For all the matrices with $a_1,a_2\geq 2$ that don't appear in the table (they are $17$) we have either $Z_A=\{(a_1,a_2)\}$ (for $16$ of them, all of which have $m>1$) or $A$ is the only configuration matrix of a CICY in $\PP^2\times\PP^2$. For this configuration one has 
$$Z_{A_{\PP^2\times \PP^2}}=\{2\leq d_i\leq 5\}\setminus \{(5,5)\}.$$
They were not included in the table in order to keep it more readable.

\end{document}